\documentclass[12pt]{article}
\usepackage{amsmath,amssymb}
\usepackage{amsthm}
\usepackage{graphicx,tikz}
\usepackage{enumerate}
\usepackage{titlesec}
\usepackage{epstopdf}
\usepackage{caption}
\usepackage{xcolor}

\numberwithin{equation}{section}
\newtheorem{Thm}{Theorem}[section]
\newtheorem*{Thm*}{Theorem}
\newtheorem{Prop}[Thm]{Proposition}
\newtheorem{Lem}[Thm]{Lemma}

\newtheorem{Cor}[Thm]{Corollary}
\newtheorem{Fact}[Thm]{Fact}

\newtheorem{Exa}[Thm]{Example}

\theoremstyle{remark}

\newtheorem{Rem}[Thm]{Remark}

\theoremstyle{definition}

\newtheorem{Def}[Thm]{Definition}

\newtheorem*{Def*}{Definition}

\numberwithin{equation}{section}

\newcommand{\g}[1]{{\mathfrak #1}}
\newcommand{\m}[1]{\mathbb{ #1}}
\newcommand{\mc}[1]{\mathcal{ #1}}

   \def\wt{\widetilde}
\def\al{\alpha}               
\def\de{\delta}       \def\eps{\varepsilon}  
       \def\la{\lambda}      
\def\si{\sigma}                
\def\ph{\varphi}               
              
\def\La{\Lambda}

\theoremstyle{definition}

\theoremstyle{remark}

\newtheorem{Rmq}[Thm]{Remark}

\numberwithin{equation}{section}
\newfont{\goth}{eufm10 at 12pt}
\newfont{\gots}{eufm8 at 9pt}

\def\bt{\begin{Thm}}
\def\et{\end{Thm}}
\def\br{\begin{Rmq}}
\def\er{\end{Rmq}}

\def\bc{\begin{Cor}}
\def\ec{\end{Cor}}
\def\bp{\begin{Prop}}
\def\ep{\end{Prop}}
\def\bl{\begin{Lem}}
\def\el{\end{Lem}}
\def\bd{\begin{Def}}
\def\ed{\end{Def}}
\def\bq{\begin{quotation}}
\def\eq{\end{quotation}}
\def\bfa{\begin{Fact}}
\def\efa{\end{Fact}}
\def\bexa{\begin{Exa}}
\def\eexa{\end{Exa}}

\def\ra{\rightarrow}
\def\vs{\vspace{1em}}

\begin{document}
\title{
Positive harmonic functions\\
on the Heisenberg group I
}
\author{Yves Benoist
}
\date{}

\maketitle

\begin{abstract}
\noindent 
We present the classification of positive harmonic functions 
on the Heisenberg group in the case of the southwest measure. 
\end{abstract}

{\footnotesize \tableofcontents}
\nopagebreak
\section{Introduction}


In this self-contained paper, we present the classification 
of the positive harmonic functions
on the Heisenberg group $H_3(\m Z)$  in the 
special case of the {\it southwest measure}.
This example is striking because the famous  partition
functions occur as  positive harmonic functions. 
In this case, our main result tells us that roughly 
all positive harmonic functions are 
combinations of characters and partition functions
(Theorem \ref{thharmu0}).

We will also explain with no proof how this result
can be extended to  finite positive measures on $H_3(\m Z)$
(Theorem \ref{thharmu1}).


\subsection{The partition function $p(x,y,z)$ as a potential}
\bq
We first introduce the ``partition function'' $p(x,y,z)$ 
for any integers $x$, $y$, $z$ in $\m Z$. 
\eq

\subsubsection{The partition function}
\begin{figure}[ht]
\centerline{\includegraphics[width=5cm]{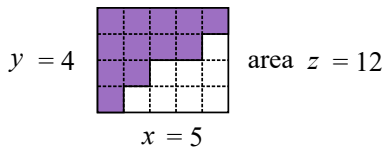}}
\caption{ {\small The partition $12\!=\!5\!+\!4\!+\!2\!+\!1$ 
is included in a $5\!\times\! 4$ rectangle.}} 
\label{figparhar}
\end{figure}
This function counts the ``number of Young diagrams of area $z$'', also called ``partitions of $z$'', 
included in a rectangle with side lengths
$x$ and $y$.
More precisely, when  $x$, $y$ and $z$ are non-negative, one has 
\begin{eqnarray}
\label{eqnpar}
p(x,y,z)&=&|\{ (n_1,\ldots , n_y)\in \m Z^y
\mid x\geq n_1\geq\cdots\geq n_y\geq 0\\
\nonumber 
&&\hspace{7em}\;\;{\rm and}\; \; n_1+\cdots +n_y=z\}|\, ,
\end{eqnarray}
and 
$p(x,y,z)=0$ otherwise. The integers $n_i$ are the lengths of the rows of the partition.
By convention, for  $x\geq 0$, one has $p(x,0,z)=0$ when $z\neq 0$,
and $p(x,0,0)=1$. 
\begin{figure}[ht]
\centerline{\includegraphics[width=12.5cm]{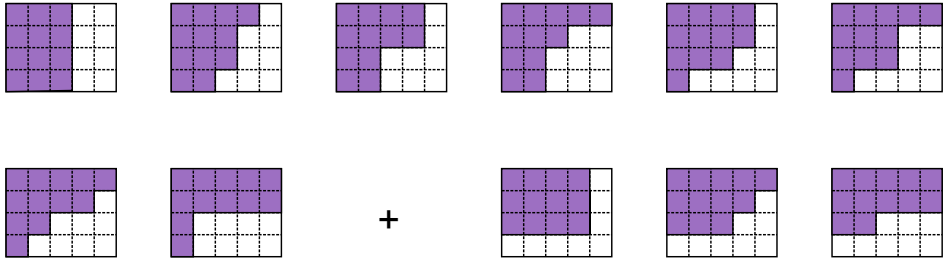}}
\caption{ {\small The $11$ partitions in  the equality
$p(5,4,12)=p(4,4,8)+p(5,3,12)$.}} 
\label{figparsum}
\end{figure}
This partition function satisfies the functional equation,
for all $g=(x,y,z)$ in $\m Z^3$, $g\neq (0,0,0)$, 
\begin{equation}
 \label{eqnpxyzha}
 p(x,y,z)=  p(x\!-\!1,y,z\!-\!y)+ p(x,y\!-\!1,z).
\end{equation}
One checks it by splitting this set of partitions according to 
the colour of the lower-left case of the rectangle as in Figure \ref{figparsum}. 

\subsubsection{The Heisenberg group}
Recall that  the Heisenberg group $G:=H_3(\m Z)$ is the set $\m Z^3$
of triples  seen as matrices 
$(x,y,z):=\mbox{\scriptsize 
	$\left(\!
	\begin{array}{ccc} 1 &x&z   \\
	0 &1&y\\
	0&0&1  
	\end{array}\!
	\right)$}.
$
It is endowed with the product  
\begin{equation}
\label{eqnprohei}
(x_0,y_0,z_0)\, (x,y,z)=  (x_0+x,y_0+y,z_0+z+x_0y)\, .
\end{equation}
Let $\mu_0$ be the southwest measure on $G$. It is given by 
\begin{equation}
\label{eqnmu0}
\mu_0=\de_{a^{-1}}+\de_{b^{-1}}
\;\;\mbox{\rm 
where $a:=(1,0,0)$ and $b=(0,1,0)$.}
\end{equation}
Let $e:=(0,0,0)$ be the unity of $G$ and ${\bf 1}_{\{e\}}$ be the characteristic function of 
$\{e\}$. Equation \eqref{eqnpxyzha} can be rewritten as, for all $g=(x,y,z)$ in $G$, 
\begin{equation}
\label{eqnpgsuphar}
p(g)=  p(a^{-1}g)+ p(b^{-1}g) +{\bf 1}_{\{e\}}(g).
\end{equation}
In particular, the function $f=p$ satisfies
\begin{equation}
\label{eqnfgsuphar}
f(g)\geq P_{\mu_0}f(g) 
\;\; \mbox{\rm where}\;\;
 P_{\mu_0}f(g) :=  f(a^{-1}g)+ f(b^{-1}g).
\end{equation}

This inequality \eqref{eqnfgsuphar} tells us that the 
function 
$f$ is a  $\mu_0$-superharmonic function on the Heisenberg group
$G$. 

\subsubsection{The potential}
More precisely,  {\it the partition function $p(g)$
is the  potential of 	$\mu_0$ at $e$}.
This means that one has the equality 
$$
p =\sum_{n\geq 0}P_{\mu_0}^n{\bf 1}_{\{e\}}
$$
\begin{figure}[ht]
\centerline{\includegraphics[width=2.5cm]{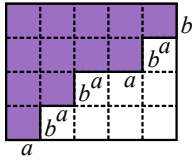}}	\caption{ {\small The partition $12\!=\!5\!+\!4\!+\!2\!+\!1$ associated to the word $w=ababaabab$\\ 
\hspace*{5em}gives  the element $g=g_w=ababaabab=(5,4,12)\in H_3(\m Z)$.}} 
\label{figparwor}
\end{figure}
Indeed, as can be seen in Figure \ref{figparwor}, for $g$ in $G$, 
\begin{equation}
\label{eqnpgnwab}
\mbox{\it p(g) is  the number of ways 
	to write $g$  as a word in
	$a$ and $b$.}
\end{equation}
A function $h$ on $G$ is said to be $\mu_0$-harmonic if it satisfies
\begin{equation}
\label{eqnhghar}
h(g)= P_{\mu_0}h(g) 
, \;\mbox{\rm for all $g$ in $G$, or equivalently}\;\;
\end{equation}
\begin{equation}
\label{eqnhxyzha}
h(x,y,z)=  h(x\!-\!1,y,z\!-\!y)+ h(x,y\!-\!1,z)
,\;\; \mbox{\rm for all $(x,y,z)$ in $\m Z^3$}.
\end{equation}

\subsection{Construction of positive $\mu_0$-harmonic functions}
\bq
We want to classify all the positive\footnote{A function $f$ on $G$ is said to be positive if $f(g)\geq 0$ for all $g$ in $G$ and $f\not= 0$.} solutions of \eqref{eqnfgsuphar}, 
i.e. all the positive $\mu_0$-superharmonic functions $h$ on $G$.
We begin with five remarks.
\eq

\subsubsection{Choquet Theorem} By a  theorem of Choquet in \cite{Choquet61}, 
every positive superharmonic function $h$ is an average
of extremal\footnote{A positive (super)harmonic function is said to be extremal if it cannot be written as the sum of two non-proportional 
positive (super)harmonic functions.}  positive superharmonic functions
$h_\al$. Moreover when $h$ is harmonic the $h_\al$ are harmonic. 
By Riesz decomposition theorem \cite[Thm 2.1.4]{Revuz84}, every positive $\mu_0$-superharmonic function
can be written in a unique way as the sum of a 
potential\footnote{A potential is a function of the form
$f=\sum_{n\geq 0}P_{\mu_{_0}}^nF$ for a positive function $F$ on $G$.}
and a positive $\mu_0$-harmonic function.
Therefore it is enough to describe 
the extremal positive $\mu_0$-harmonic functions on $G$.

\subsubsection{ Choquet-Deny Theorem} If we look for a 
$\mu_0$-harmonic function $h$ which does not depend on 
$z$, Equation \eqref{eqnhxyzha} becomes 
\begin{equation}
\label{eqnhxyhar}
h(x,y)= h(x\!-\!1,y)+ h(x,y\!-\!1)
\; ,\; \;\mbox{\rm for all $(x,y)$ in $\m Z^2$.}\;\;
\end{equation}
This equation tells us that the function $h$ is $\mu_0$-harmonic on the abelian quotient
$\m Z^2$ of $G$. 
According to a theorem of Choquet and Deny
 in \cite{ChoquetDeny60}, 
since the support  of the measure $\mu_0$ spans the group $\m Z^2$, every extremal positive $\mu_0$-harmonic function 
on this abelian group is proportionnal to a character\footnote{The proof is very short.
One notices that Equality \eqref{eqnhxyhar} gives a decomposition of $h$ as a sum of two positive harmonic functions
and hence both of them are proportional to $h$}:
\begin{equation}
\label{eqnchahar}
\chi(x,y,z)=r^xs^y
\;\;\;\mbox{\rm with $r,s>0$\; and \; 
$1/r+1/s=1$.}
\end{equation}

\subsubsection{The partition function as a harmonic function}
If we look for a 
$\mu_0$-harmonic function $h$ which does not depend on 
$x$, Equation \eqref{eqnhxyzha} becomes 
\begin{equation}
\label{eqnhyzhar}
h(y,z)= h(y,z\!-\!y)+ h(y\!-\!1,z)
\; ,\; \;\mbox{\rm for all $(y,z)$ in $\m Z^2$.}\;\;
\end{equation}
A nice example is given by the partition function 
$(y,z)\mapsto p_y(z)$ where
\begin{eqnarray}
\label{eqnh0xyz}
p_y(z)&=&\sup_{x\in \m Z} p(x,y,z) = \lim_{x\ra\infty}p(x,y,z)
=p(z,y,z)\\
\nonumber &=&
\mbox{the number of partitions of $z$ with at most $y$ rows.}
\end{eqnarray}

\begin{figure}[ht]
	\centering
	\begin{tabular}{ccccccccccc} 
		&\!$\uparrow $ z\!\!\!\! &&&&&&&&&\\
		0\!&\!\bf 0\!&\!1\!&\!5\!&\!10\!&\!15\!&\!18\!&\!20\!&\!21\!&\!\color{purple} 22\!&\!\\
		0\!&\!\bf 0\!&\!1\!&\!4\!&\!8\!&\!11\!&\!13\!&\!14\!&\!\color{purple} 15\!&\!15\!&\!\\
		0\!&\!\bf 0\!&\!1\!&\!4\!&\!7\!&\!9\!&\!10\!&\!\color{purple} 11\!&\!11\!&\!11\!&\!\\
		0\!&\!\bf 0\!&\!1\!&\!3\!&\!5\!&\!6\!&\!\color{purple} 7\!&\!7\!&\!7\!&\!7\!&\!\\
		0\!&\!\bf 0\!&\!1\!&\!3\!&\!4\!&\!\color{purple} 5\!&\!5\!&\!5\!&\!5\!&\!5\!&\!\\
		0\!&\!\bf 0\!&\!1\!&\!2\!&\!\color{purple} 3\!&\!3\!&\!3\!&\!3\!&\!3\!&\!3\!&\!\\
		0\!&\!\bf 0\!&\!1\!&\!\color{purple} 2\!&\!2\!&\!2\!&\!2\!&\!2\!&\!2\!&\!2\!&\!\\
		0\!&\!\bf 0\!&\!\color{purple} 1\!&\!1\!&\!1\!&\!1\!&\!1\!&\!1\!&\!1\!&\!1\!&\!\\
		\bf 0\!&\!\color{purple} \bf 1\!&\!\bf 1\!&\!\bf 1\!&\!\bf 1\!&\!\bf 1\!&\!\bf 1\!&\!\bf 1\!&\!\bf 1\!&\!\bf 1\!&\!$\!\rightarrow$\; y \\	
		0\!&\!\bf 0\!&\!0\!&\!0\!&\!0\!&\!0\!&\!0\!&\!0\!&\!0\!&\!0\!&\! 
	\end{tabular}	
	\caption{The function $p_y(z)$ satisfies $p_y(z)=p_y(z-y)+p_{y-1}(z)$. }
	\label{figpyz}
\end{figure}
Hence the function $h_0(x,y,z):=p_y(z)$
is a $\mu_0$-harmonic function on $G$.

\subsubsection{Margulis First Theorem}
According to the first theorem of Margulis,
a theorem he proved 
in \cite{Margulis66} when he was not yet twenty, 
Choquet-Deny Theorem is still true 
on a finitely generated nilpotent group $G$
as soon as  the support of the measure 
spans $G$ as a semigroup  (See Fact \ref{facmar}).
This is why it might look surprising at first glance, 
that there exists a
positive $\mu_0$-harmonic function $h_0$ on $H_3(\m Z)$
which is not invariant by the center. 
The reason it exists is that the support of $\mu_0$ spans 
$G$ as a group but not as a semigroup.
What is more surprising is that this ``new'' 
positive harmonic function $h_0$ is 
given by the famous partition function $p_y(z)$.

\subsubsection{Switching and translating harmonic functions} 
We denote by $\si$  the automorphism of $G$ exchanging $a$ and $b$.
It is given by 
$$
\si(x,y,z)=(y,x,xy-z)\, .
$$
Since the function $h_0$ is $\mu_0$-harmonic,
the function 
$$
h_1:= h_0\circ \si: (x,y,z)\mapsto p_x(xy-z)
$$
is also $\mu_0$-harmonic.
For $g_0$ in $G$, we denote by $\rho_{g_{_0}}:g\mapsto gg_0$
the right translation by $g_0$ on $G$.
The translated functions $h_0\circ\rho_{g_{_0}}:g\mapsto h_0(gg_0)$ 
and $h_1\circ\rho_{g_{_0}}:g\mapsto h_1(gg_0)$ are also $\mu_0$-harmonic.

\subsection{Classification of positive $\mu_0$-harmonic functions}
\bq
We can now state our main result for the southwest measure $\mu_0$ introduced in \eqref{eqnmu0}.
\eq

\subsubsection{Main result and strategy}
\bt
\label{thharmu0}
Let $h$ be an extremal  positive $\mu_0$-harmonic function on the Heisenberg group
$G:=H_3(\m Z)$. Then, up to a multiplicative scalar,\\
- either $h=\chi$ is a  $\mu_0$-harmonic character $\chi(x,y,z)=r^xs^y$ as in \eqref{eqnchahar},
\\
- or $h= h_0\circ \rho_{g_{_0}}$ is a translate of the  function $h_0(x,y,z):=p_y(z)$, \\
- or $h= h_1\circ \rho_{g_{_0}}$ is a translate of the  function $h_1(x,y,z):=p_x(xy-z)$.
\et

This classification  has been annouced 
on May $28^{th}$ 2019 in a short informal videotaped 
speech at the Cetraro Conference ``Dynamics of group actions''.

As we will see the partition function $p(x,y,z)$ will play a crucial role in the proof of Theorem \ref{thharmu0}. Indeed,
in Chapter \ref{secparfun},  we will prove a ratio limit theorem for the partition function $p(x,y,z)$. In Chapter \ref{secproof},
we will deduce from this ratio limit theorem the proof of Theorem \ref{thharmu0}.

Notice that the positive $\mu_0$-harmonic function $h_0$ vanishes.
In particular,  it does not satisfy the Harnack inequality.
This contrasts with the case studied in \cite{Margulis66}
where the support of $\mu$ spans 
$G$ as a semigroup.

In the last Section \ref{secfinsup}, we will  present the classification of the positive $\mu$-harmonic functions, 
for all finitely supported measures $\mu$ on $G$. 

\subsubsection{Dealing with a probability measure}
At first glance it might look a little bit weird to deal 
with $\mu_0$-harmonic function for a measure $\mu_0$ which is not a probability measure. 
We could have worked instead with the probability measure
$$
\wt\mu_0=\tfrac12(\de_{a^{-1}}+\de_{b^{-1}})
\;\;\mbox{\rm 
	where $a:=(1,0,0)$ and $b=(0,1,0)$}
$$
which is the law for the 
{\it southwest random walk on} $\m H_3(\m Z)$.
The $\wt\mu_0$-harmonic functions $\wt h$ on $G$ are the functions satisfying
\begin{equation*}
\label{eqnptxyzha}
\wt 
h= P_{\wt\mu_0}h
\;\; \mbox{\rm where}\;\;
P_{\wt\mu_0}h(x,y,z)=\frac12\left(\,  \wt h(x\!-\!1,y,z\!-\!y)+ \wt h(x,y\!-\!1,z)\,\right) ,
\end{equation*}
is the expected value of the function $h$ after one step of the random walk.

It is easy to see that  
\begin{equation*}
\mbox{$h(x,y,z)$ is $\mu_0$-harmonic
	if and only 
	$
	2^{-x-y}\, h(x,y,z)
	$ is $\wt\mu_0$-harmonic\, .
}
\end{equation*}
Therefore, classifying positive $\mu_0$-harmonic functions is equivalent to classifying positive $\wt\mu_0$-harmonic functions.
The main reason we are  using $\mu_0$ instead of $\wt\mu_0$ is 
to get rid of all these factors $2^{-x-y}$.

\subsubsection{Extremal superharmonic functions}
We have seen in \eqref{eqnpgsuphar} that the partition function $p$ 
is $\mu_0$-superharmonic and more precisely that it is the potential 
of $\mu_0$ at $e$.
For every $g_0$ in $G$, the function $p\circ \rho_{g_{_0}}$
is also a potential of $\mu_0$ at $g_0^{-1}$.
By Riesz decomposition Theorem, those potentials
are exactly the extremal positive $\mu_0$-superharmonic functions on $G$
which are not harmonic. Therefore, \\
\centerline{\it every extremal positive $\mu_0$-superharmonic functions $f$ on $G$
which}
\centerline{\it is not harmonic is a translate $f=p\circ \rho_{g_{_0}}$
of the function $p(x,y,z)$.}

We would like to end this introduction by pointing out
other limit theorems for random walks 
on the Heisenberg group and other nilpotent groups
as  \cite{Guivarch71}, \cite{Breuillard05}, 
\cite{Breuillard10},\cite{DiaconisHough15} eventhough we will not use them here.

\section{The partition function}
\label{secparfun}
The aim of this chapter is to prove the ratio limit theorem 
(Proposition \ref{proratlim})
for the partition function $p(x,y,z)$.

\subsection{The unimodality of the partition functions}
\label{secunipar}

\bq
We recall that, for $x, y, z\geq 0$,  the partition function $p(x,y,z)$ 
counts the number of partitions of $z$
included in a rectangle with side lengths
$x$ and $y$. 
See Definition \eqref{eqnpar} and Figure \ref{figparhar}.
\eq

This function 
is non-zero for $0\leq z\leq xy$ and satisfies the equalities 
\begin{equation}
\label{eqnpxypyx}
p(x,y,z)=p(y,x,z)=p(x,y,xy-z). 
\end{equation}
This function is well-studied. For instance one has
\begin{Fact}
\label{facparuni}
{\bf (Cayley, Sylvester 1850)}	
The sequence $z\mapsto p(x,y,z)$ 
is unimodal, i.e. it is increasing for $z\leq xy/2$.
\end{Fact}

The proof below relies on the theory of finite dimensional representations of
the Lie algebra $\g s\g l(2,\m R)$. This proof  is due to Hughes in \cite{Hughes77}.
See \cite{Proctor82} for an elementary proof and 
\cite[p. 522]{Stanley89} for a survey of various generalizations.

\begin{proof}[Sketch of proof of Fact \ref{facparuni}]
Let $n:= x+y$ and $(Y,H,X)$ be the principal $\g s\g l_2$-triple in the Lie algebra
$\g g:=\g s\g l(n,\m R)$ so that
$H={\rm diag}(n\!-\!1,n\!-\! 3,\ldots,-n\!+\!1)$.
This Lie algebra $\g g$ has a natural representation 
in the space $V:=\La^x(\m R^{n})$. One checks that $p(x,y,z)= \dim V_{xy-2z}$ 
where $V_\la$ denotes the eigenspace of $H$ in $V$ for the eigenvalue $\la$.
The theory of representations of $\g s\g l(2,\m R)$ tells us that for $\la> 0$, 
one always has $\dim V_\la\leq \dim V_{\la-2}$.
\end{proof}

\subsection{The ratio limit theorem}
\label{secratlim}
\bq
Here is the
{\it Ratio Limit Theorem for $p(x,y,z)$}:
\eq

\bp
\label{proratlim} 
One has 
$
\displaystyle 
\lim_{\substack{z\ra \infty\\ xy\!-\!z\ra\infty}} \frac{p(x,y,z-1)}{p(x,y,z)}
\;=\;1.
$
\ep
This limit is taken along sequences  of positive triples
$(x,y,z)$ such that $z\ra \infty$ and $xy\!-\!z\ra\infty$.

With this generality this theorem seems to be new, eventhough there already exist
very precise estimates of $p(x,y,z)$ in certain ranges.
For instance, when $x,y\geq z$, the partition function $p(x,y,z)=p(z,z,z)$ depends only on $z$. 
It is the classical partition function  $p(z)$
which admits a famous asymptotic expansion 
due to Hardy and Ramanujan in 1920
(See \cite[Ch. 5]{Andrews76}).
These estimates have been extended to larger ranges
of $(x,y,z)$ 
as in 
\cite{Takacs86} and \cite{MePaPe18}. 
We will not use them.
\vs

The proof of Proposition \ref{proratlim} is tricky but elementary.
The rough idea is to introduce a relation 
between the set of partitions $w$ of $z$ and the set of partitions $w'$ of $z-1$
such that ``most of the time'' when $w$ and $w'$ are related,
they are related to approximately
the same number of partitions (see Lemma \ref{lemwwfwfw}). 

Because of \eqref{eqnpxypyx}, we can assume that $y\leq x$ 
and $z\leq xy/2$.

\subsection{When the height of the rectangles is bounded}
\bq
In this section, we  deal with the easy case when the height
$y$ remains bounded.
\eq

\bl
\label{lemratlim}  For all $y\geq 1$, one has
$
\displaystyle 
\lim_{\substack{x,z\ra \infty\\ z\,\leq \,xy/2}} \frac{p(x,y,z-1)}{p(x,y,z)}=1.
$
\el
Note that in this limit $y$ is fixed, and 
$x$, $z$ go to $\infty$ with $z\leq xy/2$.

\begin{proof}[Proof of Lemma \ref{lemratlim}]
This follows from  Lemma \ref{lemratlim1} and the inequalities
$$
0\;\leq\; 
p(x,y,z)-p(x,y,z-1) 
\;\leq\; 
p(x,y-1,z).
$$

The first inequality is the unimodality of the partition function. 

For the second inequality, just notice 
that one can inject the set of partitions of $z$ of height exactly $y$ 
inside the set of partitions of $z-1$ of 
height at most $y$ by removing the last square in the bottom row
of each partition.
\end{proof}

We have used the following Lemma.
\bl
\label{lemratlim1} $a)$ For all $x,y,z\geq 1$, one has 
$\;\; p(x,y,z) \leq   z^{y-1}$.\\
$b)$
For all $y\geq 1$, there exists $\al_{{y}}>0$ such that,
for all $x,z\geq 1$ with $z\leq xy/2$, one has
$\;\; p(x,y,z) \geq  \al_y\, z^{y-1}.$
\el

\begin{proof}[Proof of Lemma \ref{lemratlim1}]
$a)$ The lengths of the 
last $y\!-\!1$ rows of the partition are bounded by $z-1$ and 
the first row is deduced from the others.

$b)$ Choose $y\!-\! 1$
integers $m_1,..,m_{y\!-\!1}$ in the interval $[0,\frac{z}{y^2}]$.
and keep only those for which the system
$$
n_1-n_2=m_1\;,\; \ldots\; ,\;\; n_{y-1}-n_y=m_{y\!-\!1}\;\;{\rm and}\;\; n_1+\cdots +n_y=z
$$ 
has a solution $(n_1,\ldots ,n_y)$ in $\m Z^y$.  
But then one has 
$$
n_y= \frac{1}{y}\,(z-m_1-2m_2-\cdots -(y\!-\!1)m_{y\!-\!1})\geq 0\;\;\; {\rm and}
$$
$$
n_1= n_y +m_1+\cdots +m_{y\!-\!1}\leq \frac{z}{y}+\frac{z}{y}\leq x.
$$
This gives about $\frac1y (\frac{z}{y^2})^{y-1}$ partitions of $z$ with  $x\geq n_1\geq \cdots\geq n_y\geq 0$.
\end{proof}

\subsection{Inner and outer corner of a partition}
\bq 
We now introduce notations that will stengthen the connection 
between partitions and words in the Heisenberg group.
\eq

We recall that $a=(1,0,0)$ and $b=(0,1,0)$ are the generators of the Heisenberg group $G=H_3(\m Z)$.
Let  
$$
G^+:=\{g=(x,y,z)\in G \mid x,y\geq 0
\;\;{\rm and}\;\; 0\leq z\leq xy\}
$$ 
be the semigroup generated by $a$ and $b$
and let 
\begin{equation}
\label{eqngaabab}
c=aba^{-1}b^{-1}=(0,0,1) 
\end{equation}
be the generator of the center $Z$ of $G$.

Let $B_n:=\{a,b\}^n$ be the set of finite words $w$ in $a$, $b$ of length 
$\ell_w=n$ and let $B:=\cup_{n\geq 0}B_n$. 
Using the product law in $G$, to each word $w\in B$,
we can associate an element $g_w$ in $G^+$. 
The partition function gives the size of the fibers of this map~:
\begin{equation}
\label{eqnpgw}
p(g)=|B_g|
\;\;{\rm where}\;\; \;
B_g:=\{w\in B\mid g_w=g\}.
\end{equation}
Indeed, as explained in Figure \ref{figparwor},
when $g=(x,y,z)$, each word $w$ in $B_g$
corresponds uniquely to a partition of $z$ 
included in a rectangle with side lengths
$x$ and $y$.
We introduce now the following relation $\mc R$ on $B$,
\begin{eqnarray*}
\label{eqnrww}
\mc R&:=&
\{ (w,w')\in B\times B\mid 
w=w_0ab w_1
\;{\rm and}\;  w'=w_0ba w_1\\
\nonumber&&
\hspace*{10em} \mbox{\rm for some}\; w_0, w_1 \; {\rm in}\; B\}.
\end{eqnarray*}
Let $\pi:\mc R\ra B$ and $\pi':\mc R\ra B$ be the two projections
$$
\pi(w,w')=w
\;\;{\rm and}\;\;
\pi(w,w')=w'.
$$
For $w$, $w'$ in $B$, the cardinality of the fiber $f_w:=|\pi^{-1}(w)|$ is the number of pairs $ab$ occuring in the word $w$. The size $f_w$ is also the number of {\it inner corners}
of the partition associated to $w$. See Figure \ref{figparpar}.
Similarly the cardinality of the fiber $f'_{w'}:=|\pi'^{-1}(w')|$ is the number of pairs $ba$ occuring in the word $w'$. It is equal  to the number of {\it outer corners}
of the partition associated to $w'$.

\begin{figure}[ht]
\centerline{\includegraphics[width=10cm]{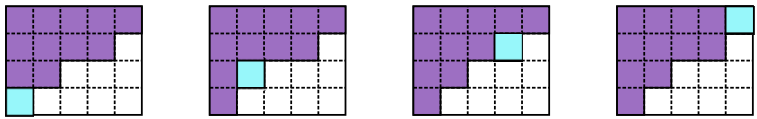}}
\caption{ {\small The fiber $\pi^{-1}(w)$ of the word $w=ababaabab$ has size $f_w=4$.}} 
\label{figparpar}
\end{figure}

The following lemma compares the size of these fibers.

\bl
\label{lemwwfwfw}
$a)$ For all $(w,w')\in \mc R$, 
one has $g_w=g_{w'}c$.\\
$b)$ For all $(w,w')\in \mc R$, 
one has $|f_w-f'_{w'}|\leq 2$.\\
In particular, one also has $f_{w}\leq 3f'_{w'}$.
\el
\begin{proof}[Proof of Lemma \ref{lemwwfwfw}]
$a)$  This follows from the equality $c=aba^{-1}b^{-1}$.

$b)$ Comparing the number of pairs $ab$ and pairs $ba$
occuring in $w$ and in $w'$, one gets $|f_w-f_{w'}|\leq 1\;$ and 
$\;|f_{w'}-f'_{w'}|\leq 1$.
\end{proof}

\subsection{Partitions with bounded number of corners}
\label{secinncor}
\bq
We will need to control  the number $p_{_{\leq i}}(x,y,z)$ of partitions of $z$
included in a rectangle with side length $x$, $y$ that have at most $i$
inner corner.
\eq

The following Lemma \ref{lembgi}
tells us that  $p_{_{\leq i}}(x,y,z)$ is negligeable 
compared to the total number of partitions $p(x,y,z)$.

\bl
\label{lembgi} 
For all $i\geq 0$, one has 
$\;\;
\displaystyle \lim_{\substack{x,y,z\ra \infty\\ z\,\leq \,xy/2}} 
\frac{p_{_{\leq i}}(x,y,z)}{p(x,y,z)}=0.
$
\el

The limit is taken along sequences 
where  all coordinates $x$, $y$, $z$ go to $\infty$
and $z\leq xy/2$.

\begin{proof}[Proof of Lemma \ref{lembgi}]
Use the following slight upgrade of Lemma \ref{lemratlim1}.	
\end{proof}

\bl
\label{lemratlim2} 
$a)$ For all $x,y,z,i\geq 1$, one has 
$\;\; p_{_{\leq i}}(x,y,z) \leq   (2z)^{2i}$.\\
$b)$
For all $j> 1$, there exists $z_0=z_0(j)\geq 1$ such that,
for all $x,y,z\geq 1$ with $4\, j \leq y\leq x$ and $z_0\leq z\leq xy/2$, one has 
$\;\; p(x,y,z) \geq   z^{j}.$
\el

\begin{proof}[Proof of Lemma \ref{lemratlim2}] It is similar to Lemma \ref{lemratlim1}.
	
$a)$ We can assume $x=y=z$. We want to choose integers $a_1,\ldots, a_i\geq 1$
and $m_1,\ldots, m_i\geq 0$, bounded by $z$ such that $a_1m_1+\cdots+a_im_i=z$.
There are at most $(2z)^{2i}$ possibilities.
	
$b)$ We give a rough count. 
Choose $L_y\leq y$ as large as possible such that, setting $\ell_y=[L_y/2]$ and
$\ell_x=[z/L_y]$, one has $\ell_y\leq \ell_x\leq x/2$. 
There exists a partition $w_0$
of $z$ with $L_y$ rows and all of whose rows have length $\ell_x$ or $\ell_x\!+\!1$.
For every sequence $\ell_x>m_1\geq\cdots\geq m_{\ell_{_y}}\geq 0$,
we can modify this partition $w_0$ by adding $m_j$ spots  to the $j^{\rm th}$ highest row of $w_0$
and removing $m_j$ spots to the $j^{\rm th}$ lowest row of $w_0$, for all $j\leq \ell_y$.
This gives  $N$ different partitions of $z$ where 
$N:=\binom{\ell_x+\ell_y-1}{\ell_y}\geq \max(2,\ell_x/\ell_y)^{\ell_y} $.
Hence, one has $p(x,y,z)\geq N$.	

{\bf First case} : when $z\leq y^2/2$. \; In this case, we have $L_y= [\sqrt{2z}]$.

One gets $ N\geq 2^{\ell_y}\geq 2^{\sqrt{z}/2}\geq z^j$.

{\bf Second case} : when $z\geq y^2/2$. \; In this case, we have $L_y=y$.

If $z\leq y^4$,
one gets $N\geq 2^{\ell_y}\geq 2^{\sqrt[4]{z}/4}\geq z^j$.

If $z\geq y^4$,
one gets 
$N\geq (\frac{\ell_x}{\ell_y})^{\ell_y}
\geq (\frac{z}{y^2})^{\ell_y}\geq {\sqrt{z}}^{\,\ell_y}\geq z^{y/4}\geq z^j$.
\end{proof}

\subsection{When the height of the rectangles is unbounded}
\bq 
We can now explain the proof of the ratio limit theorem. 
\eq

\begin{proof}[Proof of Proposition \ref{proratlim}]
By \eqref{eqnpxypyx} and Lemma \ref{lemratlim}, we can  assume that 
$x,y,z$  are going to $\infty$ with $y\leq x$ and $z\leq xy/2$.
For $g=(x,y,z)$ in $G^+$, one sets 
$
\mc R_g:=\{(w,w')\in \mc R\mid g_w=g\},
$
and one computes
\begin{eqnarray}
\label{eqnpgbg}
p(g)=|B_g|=\eps_g +\sum_{(w,w')\in \mc R_g}\frac{1}{f_w} 
\end{eqnarray}
where $\eps_g =1$ if $\mc R_g=\emptyset$ and 
$\eps_g=0$ otherwise.
Similarly, by Lemma \ref{lemwwfwfw}.$a$, one has 
\begin{eqnarray}
\label{eqnpggabg}
p(gc^{-1})=|B_{gc^{-1}}|=\eps'_g +\sum_{(w,w')\in \mc R_g}\frac{1}{f'_{w'}} 
\end{eqnarray}
where $\eps'_g =0$ or $1$.
Combining \eqref{eqnpgbg}, \eqref{eqnpggabg} and Lemma \ref{lemwwfwfw}.$b$, one gets
\begin{eqnarray*}
\label{eqnpgpgga}
|p(g)\!-\! p(gc^{-1})|
&\leq& 
2+\!\!\sum_{(w,w')\in \mc R_g}\!\!\frac{2}{f_wf'_{w'}}
\;\; \leq\;\;
2+\!\! \sum_{(w,w')\in \mc R_g}\frac{6}{f_w^2}
\;\;\leq\;\; 
2+\!\sum_{\substack{w\in B_g\\f_w\neq 0}}\frac{6}{f_w}
\end{eqnarray*}
We recall that $p_{_{\leq i}}(g)$ is the number of $w$ 
with $f_w\leq i$. Therefore, one has
\begin{eqnarray*}
|p(g)-p(gc^{-1})|
&\leq& 
2+6\, p_{_{\leq i}}(g)+\frac{6}{i}\, p(g)
\;\;\;\;\mbox{\rm for all $i\geq 1$}.
\end{eqnarray*}
We let $x,y,z$ go to infinity with $z\leq xy/2$.
According to Lemma \ref{lembgi}, for all $i\geq 1$, the ratios $p_{_{\leq i}}(g)/p(g)$ converge to $0$. 
Therefore 
\begin{eqnarray*}
\label{eqnpgpgga2}
\limsup \left|\frac{p(gc^{-1})}{p(g)}-1\right|
&\leq&  \frac{6}{i}.
\end{eqnarray*}
and therefore the sequence $\displaystyle\frac{p(gc^{-1})}{p(g)}$ converges to $1$ as required.
\end{proof}

\section{Positive harmonic functions}
\label{secproof}

We now start the classification of extremal positive  $\mu_0$-harmonic
functions $h$. 
In Section \ref{secpyzhar}, we deal with the case where $h$ has a non-zero limit along an orbit of $a^{-1}$ or $b^{-1}$.
In Sections \ref{sechardec} and \ref{secusirat},  we deal with the case where $h$ goes to zero along all orbits of $a^{-1}$ and $b^{-1}$.
In Section \ref{secfinsup} we present the generalization of this classification
to any finitely supported measure $\mu$ on $G$.

\subsection{The function $p_y(z)$ as an harmonic function}
\label{secpyzhar}
\bq
In this section we characterize the functions $h_0\circ \rho_{g_{_0}}$ and $h_1\circ \rho_{g_{_0}}$
among extremal positive $\mu_0$-harmonic functions
by their behavior along the orbits $a^{-\m N}g_0$ and $b^{-\m N}g_0$ of $G$.
\eq

We recall that $a=(1,0,0)$ and $b=(0,1,0)$ are the generators of the Heisenberg group $G=H_3(\m Z)$, that 
$\mu_0=\de_{a^{-1}} +\de_{b^{-1}}$,
and  that $h_0$ and $h_1$ are 
the $\mu_0$-harmonic functions 
$h_0(x,y,z)=p_y(z)$ and $h_1(x,y,z)=p_x(xy-z)$. 

We first begin by an alternative construction of the function $h_0$.
Let $H_0$ be the abelian subgroup of $G$ generated by $a$
and let $\psi_0:={\bf 1}_{H_0}$ be the characteristic function of $H_0$.
One has 
\begin{eqnarray*} 
	\psi_0(x,y,z)&=1& \mbox{when $y=z=0$}\\
	&=0& \mbox{otherwise.}
\end{eqnarray*}

\bl
\label{lemlimpmunh0} One has the equality  
$\displaystyle 
\;\; h_0\;=\;\lim_{n\ra\infty} P^n_{\mu_0}\psi_0 .
$
\el

\begin{Rem}
Since the function $\psi_0$ is $\mu_0$-subharmonic, i.e.
$\;
\psi_0\leq P_{\mu_0}\psi_0\, ,
$
the sequence $n\mapsto  P_{\mu_0}^n\psi_0$ is increasing.
\end{Rem}
\begin{proof}[Proof of Lemma \ref{lemlimpmunh0}] One can compute explicitely
this function  $P_{\mu_0}^n\psi_0$. It does not depend on $x$.
Indeed $P_{\mu_0}^n\psi_0(x,y,z)$ is the number of ways of writing 
the element $(n-y,y,z)$ as a word $w$ of length $n$ in $a$ and $b$.
This proves the equality, involving the partition function, 
$$
P_{\mu_0}^n\psi_0(x,y,z)=p(n-y,y,z)\,.
$$
Letting $n$ go to $\infty$, we conclude using \eqref{eqnh0xyz}.
\end{proof}

\begin{Lem}
\label{lemharpyz}
Let $g_0\in G$ and $h$ be an extremal positive $\mu_0$-harmonic function on $G$ such that 
$\displaystyle\limsup_{n\ra\infty}h(a^{-n}g_0)>0$. Then one has 
$h=\la\, h_0\circ \rho_{g_{_0}}$ with $\la>0$.

In particular, the positive $\mu_0$-harmonic function 
$h_0\circ \rho_{g_{_0}}$ is extremal.
\end{Lem}

\begin{proof}[Proof of Lemma \ref{lemharpyz}]
	We can assume that $g_0=e$.
Since the function $h$ is positive and $\mu_0$-harmonic, the sequence
$n\mapsto h(a^{-n})$ is decreasing. Hence it has a limit $\la$. 
By assumption, this limit $\la$ is positive. By construction,
one has the equality 
$h\geq \la \psi_0$. Since $h$ is $\mu_0$-harmonic, one also has the inequality
$h\geq \la P^n_{\mu_0}\psi_0$ for all $n\geq 0$. Therefore, by Lemma \ref{lemlimpmunh0},
one gets 
$h\geq \la h_0$. 
Since $h$ is extremal, it has to be proportional to $h_0$ 
and therefore one has $h=\la h_0$.

It remains to check that $h_0$ is extremal. 
If one can write $h_0=h'_0+h''_0$ with both $h'_0$ and $h''_0$ 
positive $\mu_0$-harmonic, for at least one of them, say $h'_0$,
the sequence $h'_0(a^{-n})$ does not converge to $0$ for $n\ra\infty$. 
Hence, by the previous discussion, $h'_0$ is proportional to $h_0$.
This proves that $h_0$ is extremal.
\end{proof}

Exchanging the role of $a$ and $b$ we get
\begin{Cor}
	\label{corharpyz}
	Let $h$ be an extremal positive $\mu_0$-harmonic function on $G$ such that 
	$\displaystyle\limsup_{n\ra\infty}h(b^{-n}g_0)>0$. Then one has $h=\la \, h_1\circ \rho_{g_{_0}}$ for some $\la>0$.
	
	In particular, the positive $\mu_0$-harmonic function 
	$h_1\circ \rho_{g_{_0}}$ is extremal.
\end{Cor}

\subsection{Harmonic functions  that decay on cosets}
\label{sechardec}
\bq
We now discuss positive harmonic functions on $G$ that decay to $0$ 
along the orbits $a^{-\m N} g_0$
and $b^{-\m N} g_0$.
\eq

Let $G_n$ be the subset of $G$ consisting of elements of 
``degree'' $n$,
\begin{eqnarray*}
G_n&=&\{g=(x,y,z)\in G\mid 
x\! +\! y\! =\! n\}.
\end{eqnarray*}
By definition and by \eqref{eqnpgnwab}, a positive $\mu_0$-harmonic function 
$h$ on $G$ satisfies the equality, for all $n\geq 1$,
\begin{equation}
\label{eqnhg0spg}
h(g_0)=\sum_{w\in B_{n}}\,  h(g_w^{-1}g_0)=\sum_{g\in G_{n}}\, p(g)\, h(g^{-1}g_0)\, .
\end{equation}
For an integer $A>0$, we set 
\begin{eqnarray}
\label{eqngnaxyz}
G_{n,A}&=&\{g=(x,y,z)\in G_n\mid z\leq A\},\\
\nonumber
G^\si_{n,A}&=&\{g=(x,y,z)\in G_n\mid 
xy\!-\! z\leq A\}.
\end{eqnarray} 
The following lemma tells us when the contributions 
of $G_{n,A}$ and $G^\si_{n,A}$ in Formula 
\eqref{eqnhg0spg} is negligeable.

\bl
\label{lemsumgna}
Let $h$ be a positive $\mu_0$-harmonic function on $G$ such that, 
\begin{equation}
\label{eqnlimalimb}
\lim_{n\ra\infty}h(a^{-n}g_0)=0
\;\; {\rm and}\;\; \lim_{n\ra\infty}h(b^{-n}g_0)=0
\;\;\;\mbox{\rm for all $g_0$ in $G$.}
\end{equation} 
Then, for all $A> 0$ and $g_0$ in $G$, one has 
\begin{equation}
\label{eqnlimsumhgg}
\lim_{n\ra\infty}\sum_{g\in G_{n,A}\cup G^\si_{n,A}}\, p(g)\, h(g^{-1}g_0) =0\, .
\end{equation}
\el

\begin{proof}[Proof of Lemma \ref{lemsumgna}]
It is enough to prove \eqref{eqnlimsumhgg} with $g_0=e$.
Moreover, since $G^\si_{n,A}$ is the image of $G_{n,A}$ by the involution $\si$  which exchanges $a$ and $b$, 
it is enough to prove \eqref{eqnlimsumhgg}
with $g\in G_{n,A}$.
Equivalently, it is enough to prove 
\begin{equation}
\label{eqnlimsumhgw}
\lim_{n\ra\infty}\sum_{w\in B_{n,A}}\, h(g_w^{-1}) =0
\;\;\;\;{\rm where}\;\;\;
B_{n,A}:= \{ w\in B_n\mid g_w^{-1}\in G_{n,A}\}.
\end{equation}
Note that, when $n>A$, every 
word $w\in B_{n,A}$ can be written as 
$$
w=b^m sa^k
$$ with $s\in B_{A+1}$ a word of length  $A\! +\! 1$. See Figure \ref{figbetasalpha}.
\begin{figure}[ht]
	\centerline{\includegraphics[width=4cm]{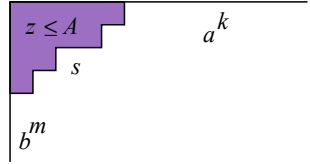}}
	\caption{ {\small The decomposition $w=b^m sa^k$
	for a word $w\in B_{n,A}$.}} 
	\label{figbetasalpha}
\end{figure}
One splits the set  $B_{n,A}$ according to $m\geq A$
or $m< A$. 
Therefore, for $n\geq 2A$, one has the inclusion
\begin{eqnarray*}
B_{n,A}\subset 
b^{A}\,B_{n-\!A}\;\cup\; 
B_{2A}\, a^{n-2A}.
\end{eqnarray*}
Therefore, using  \eqref{eqnhg0spg}, one gets the inequalities
\begin{eqnarray*}
\sum_{w\in B_{n,A}}\,  h(g_w^{-1}) 
&\leq & 
\sum_{w\in B_{n-A}}h(g_w^{-1}b^{-A})\; +\; 
\sum_{w\in B_{2A}} h(a^{-(n-2A)}g_w^{-1})\\
&= & 
h(b^{-A})\; +\; 
\sum_{w\in B_{2A}} h(a^{-(n-2A)}g_w^{-1})
\end{eqnarray*}
For all $\eps >0$,
we choose $A$ large enough so that, by the second assumption \eqref{eqnlimalimb}, one has  $h(b^{-A})\leq \eps$.
Then the last sum is a  sum over the fixed finite set $B_{2A}$, 
and, by the first  assumption \eqref{eqnlimalimb}, 
this last sum converges to $0$ when $n$ goes to infinity.
This proves \eqref{eqnlimsumhgw} as required. 
\end{proof}

\subsection{Using the ratio limit theorem}
\label{secusirat}
\bq
Combining Lemma \ref{lemsumgna} with the ratio limit theorem 
we can finish the last case
of the proof of Theorem \ref{thharmu0}.
\eq
\bl
\label{lemhaghbg}
Let $h$ be a positive $\mu_0$-harmonic function on $G$ such that, for all $g_0$ in $G$, 
$\displaystyle\lim_{n\ra\infty}h(a^{-n}g_0)=
\lim_{n\ra\infty}h(b^{-n}g_0)=0$. 
Then $h$ is invariant by the center $Z=c^\m Z$ of $G$.
\el

\begin{proof}[Proof of Lemma \ref{lemhaghbg}] 
Using \eqref{eqnhg0spg} with $g_0$ and $g_0c$, 
we compute,
\begin{eqnarray}
\label{eqnhg0hgg}
h(g_0)-h(g_0c)
&=& \sum_{g\in G_{n}}\, (p(g)- p(gc))\, h(g^{-1}g_0).
\end{eqnarray}

We fix $\eps>0$. According to the ratio limit theorem
(Proposition \ref{proratlim}),
there exists an integer $A>0$ such that, 
for all $g=(x,y,z)$ in $G^+$
with $z\geq A$ and $xy-z\geq A$,
one has 
\begin{eqnarray}
\label{eqnpgpggm}
|p(g)- p(gc)|\leq \eps\, p(g)\, .
\end{eqnarray}
Therefore, using \eqref{eqnhg0hgg}, \eqref{eqnpgpggm} and
Definition \eqref{eqngnaxyz}, one gets
\begin{equation*}
|h(g_0)-h(g_0c)|
\leq \sum_{g\in G_{n}}\, \eps\, p(g)\, h(g^{-1}g_0)
+ \!\!\sum_{g\in G_{n,A}\cup G^\si_{n,A}}
\!\!\!
p(g)(h(g^{-1}g_0) + h(g^{-1} g_0c))
\end{equation*}
By \eqref{eqnhg0spg}, the first term is equal to $\eps h(g_0)$. 
Therefore using twice Lemma \ref{lemsumgna} and letting $n$ go to infinity, one gets
$|h(g_0)-h(g_0c)|\leq \eps\, h(g_0)$.
Since $\eps$ is arbitrary small,
this proves that $h(g_0)=h(g_0c)$ as required.
\end{proof}

\bc
\label{corhaghbg}
Let $h$ be an extremal positive $\mu_0$-harmonic function on $G$ such that, for all $g_0$ in $G$, 
$\displaystyle\lim_{n\ra\infty}h(a^{-n}g_0)=
\lim_{n\ra\infty}h(b^{-n}g_0)=0$. 
Then $h$ is a character of $G$.

In particular, every $\mu_0$-harmonic character of $G$ is an extremal
positive $\mu_0$-harmonic function.
\ec

\begin{proof}[Proof of Corollary \ref{corhaghbg}]
By
Lemma \ref{lemhaghbg}, the function $h$ is $\mu_0$-harmonic on the abelian group $G/Z$. 
By Choquet-Deny Theorem, it is a character.

It remains to check that a $\mu_0$-harmonic character $\chi$ is extremal.
Assume that $\chi=h'+h''$ with both $h'$ and $h''$ 
positive $\mu_0$-harmonic. For all $g_0$ in $G$,
the sequences $h'(a^{-n}g_0)$ and $h'(b^{-n}g_0)$ converge to $0$ for $n\ra\infty$. 
Hence, by the previous discussion and by Choquet Theorem, the function $h'$ is an integral 
$h'=\int_{C} \chi' {\rm d}\si(\chi')$
where $\si$ is a finite positive measure on the 
set $C$ of (harmonic) character $\chi'$ of $G$. Since $h'\leq \chi$, 
the measure $\si$ must be supported by $\chi$.
This proves that $\chi$ is extremal.
\end{proof}

This ends the proof of Theorem \ref{thharmu0}.

\subsection{Extension to finitely supported measures}
\label{secfinsup}
\bq
In this section we give 
the classification of the positive $\mu$-harmonic functions on the Heisenberg group for all finitely supported measure $\mu$.
\eq

Let $G=H_3(\m Z)$ be the Heisenberg group
and $S$ be a finite subset of $G$.
We denote by $G_S$ the subgroup of $G$ generated by $S$.
Let $\mu=\sum_{s\in S} \mu_s \de_s$ be a positive measure on $G$ with support $S$.

We recall that a function $h$ on $G$ is said to be $\mu$-harmonic
if
\begin{equation}
\label{eqnhhar}
h=P_\mu h
\;\;{\rm where}\;\;
P_\mu h(g):= \textstyle \sum_{s\in S} \,\mu_s\, h(sg). 
\end{equation}

We want to describe the cone $\mc H^+$ of positive $\mu$-harmonic functions $h$ on $G$.
By Choquet Theorem, it is enough to describe the extremal rays of this cone $\mc H^+$.

There are two constructions of extremal positive $\mu$-harmonic functions.

\subsubsection{The harmonic characters $\chi$}
The $\mu$-harmonic characters are the 
characters 
$\chi: G\ra \m R_{>0}$ of $G$ such that $\sum_{s\in S}\mu_s\,\chi(s)=1$. 
Such a function
$h=\chi$ is an
extremal positive $\mu$-harmonic function on $G$ which is invariant by the center $Z$ of $G$.

We now recall Margulis Theorem which tells us that this first construction is the only possible when $G_\mu^+=G$.

\begin{Fact}
\label{facmar}{\bf (Margulis)}
Let $\mu$ be a finite positive measure on a finitely generated nilpotent group~$G$. 
If the semigroup  $G_\mu^+$ generated by the support of $\mu$ is equal to $G$, 
then every extremal positive $\mu$-harmonic function $h$ on $G$ is a character.
\end{Fact}

\begin {proof}[Sketch of proof of Fact \ref{facmar} for $G=H_3(\m Z)$]
Because of the assumption $G_\mu^+=G$, we can assume that $\mu_c>0$ and $\mu_a>0$.
The first part of the argument is as in the abelian case~: 
since $\; h(x,y,z)\geq \mu_c\, h(x,y,z+1)$,
these two \mbox{$\mu$-harmonic} functions are proportional and  we get that, for some $t>0$, one has $h(x,y,z)=h(x,y,0)t^z$.
We now want to prove that  $t=1$.
 
Let $K_t$ be the set of positive harmonic functions $h_0(x,y,z)=\psi_0(x,y)t^z$ with $h_0(e)=1$. 
Since $G_\mu^+=G$, the convex set $K_t$ is  compact 
for the pointwise convergence. The element $a\in G$ acts continuously by ``right-translation and renormalization'' on $K_t$.
By Schauder fixed point theorem, this action  has a fixed point $h_0$ in $K_t$. It
can be written as $h_0(x,y,z)=r^x\ph_0(y)t^z$ with $r>0$.
But then one writes $h_0(g)\geq \mu_ah_0(ag)$ for all $g$ in $G$,  or equivalently $\ph_0(y)\geq \mu_ar\ph_0(y)t^{y}$ for all $y\in \m Z$. This proves that $t=1$.
\end{proof}

When $G_\mu^+\neq G$, a second construction is possible.

\subsubsection{The functions $h_{S_{_0},\chi_{_0}}$
induced from a harmonic character}
Let $S_0\subset S$ be an abelian subset. Denote by 
$\mu_{S_{_0}}:=\sum_{s\in S_{_0}}\mu_s\, \de_s$ the measure restriction of $\mu$ to $S_0$.
Let $\chi_0$ be a $\mu_{S_{_0}}$-harmonic character of $G_{S_{_0}}$.
We extend $\chi_0$ as a function 
$$
\psi_0:=\chi_0\, {\bf 1}_{G_{S_{_0}}}
$$ 
on $G$ 
which is $0$ outside $G_{S_{_0}}$. This function $\psi_0$ is 
$\mu$-subharmonic, so that the sequence $P_\mu^n\psi_0$
is increasing. We set 
$$
h_{S_{_0},\chi_{_0}}=\lim_{n\ra\infty}P_\mu^n\psi_0.
$$
We can tell exactly for which pairs $(S_0,\chi_0)$
the function $h_{S_{_0},\chi_{_0}}$ is finite. 
In~this case the function $h_{S_{_0},\chi_{_0}}$
is an
extremal positive $\mu$-harmonic function on $G$.

We can now state the extension of Theorem \ref{thharmu0} to a more general finitely 
supported measure $\mu$ on $G$. 

\bt
\label{thharmu1}
Let $G=H_3(\m Z)$ and $\mu$ be a positive measure on $G$ 
whose finite support $S$  generates the group $G$.
Then every extremal positive $\mu$-harmonic function $h$
on $G$ is proportional either to a character $\chi$ of $G$ or 
to a translate $h_{S_{_0},\chi_{_0}}\circ \rho_{g_{_0}}$ of a function 
induced from a harmonic character. 
\et

\bc
\label{corharmu1}
Let $G=H_3(\m Z)$, $Z$ its center and $\mu$  a probability measure on $G$ whose finite support $S$ 
generates the group $G$. The following are equivalent:\\
$(i)$ Every positive $\mu$-harmonic function on $G$ is $Z$-invariant.\\
$(ii)$ $G_\mu^+$ contains two non-central elements whose product is in 
$Z\smallsetminus \{0\}$.
\ec

Theorem \ref{thharmu1} and Corollary \ref{corharmu1}
will be proven in the sequel paper \cite{BenoistHeisenberg2}.

We will also see that on the
nilpotent group of rank 4 with cyclic center, there exist extremal positive harmonic
functions which are neither an harmonic character nor a function induced from a harmonic character.

{\small
\bibliography{heisenberg}
}

{\small
\noindent
Y. \textsc{Benoist}: CNRS, 
Universit\'e Paris-Sud,\newline
e-mail: \texttt{yves.benoist@u-psud.fr}}
\end{document}